\documentclass[a4paper]{amsart}
\usepackage[textwidth=15cm,top=3cm, bottom=3cm, hcentering]{geometry}
\usepackage[colorlinks=true,linkcolor=red,citecolor=blue]{hyperref}

\usepackage{amsmath,amscd,amsthm,amssymb}
\usepackage{color}
\usepackage[utf8]{inputenc}
\normalfont
\usepackage[T1]{fontenc}
\usepackage{tikz-cd}
\usepackage{tikz}
\usepackage{enumerate}
\usepackage{enumitem}
 \usepackage[all]{xy}

\theoremstyle{theorem}
\newtheorem{theo}{Theorem}[section]

\newtheorem{prop}[theo]{Proposition}

\theoremstyle{definition}
\newtheorem{defi}[theo]{Definition}
\newtheorem{rema}[theo]{Remark}

\newcommand{\RR}{\mathbb{R}}

\newcommand{\Aa}{\mathcal{A}}
\newcommand{\Mm}{\mathcal{M}}

\newcommand{\Ll}{\mathcal{L}}
\newcommand{\Cc}{\mathcal{C}}
\newcommand{\CH}{\mathrm{CH}}
\newcommand{\Cyc}{\mathrm{CC}}
\newcommand{\dR}{\mathrm{dR}}

\newcommand{\Lie}{\mathcal{L}ie}
\newcommand{\Grav}{\mathcal{G}rav}

\title{An operadic proof of the BTT theorem}

\author{Joana Cirici}
\author{Geoffroy Horel}

\address[J. Cirici]{Departament de Matemàtiques i Informàtica, Universitat de Barcelona and
Centre de Recerca Matemàtica (Barcelona, Spain)
}
\email{jcirici@ub.edu}

\address[G. Horel]{Université Sorbonne Paris Nord, Laboratoire Analyse, Géométrie et Applications, CNRS (UMR 7539), 93430, Villetaneuse, France.}
\email{horel@math.univ-paris13.fr}

\thanks{
J. Cirici acknowledges 
Govern de Catalunya (2021-SGR-00697 and Serra H\'{u}nter Program) and the Spanish State Research Agency (CEX2020-001084-M, PID2020-117971GB-C22 and EUR2023-143450). Both authors are funded by the Agence Nationale pour la Recherche through project ANR-20-CE40-0016 HighAGT and by the Centre National pour la Recherche Scientifique through project IEA00979.
}

\begin{document}
	
\begin{abstract}
In this note, we explain an operadic proof of the BTT Theorem stating that the deformation theory of Calabi--Yau varieties is unobstructed. We also provide a short new proof of the non-commutative BTT for Calabi--Yau dg categories.
Finally, we observe that our proof also produces partial results in positive characteristic.
\end{abstract}
	
\maketitle
\tableofcontents
\section{Introduction}
The classical Bogomolov-Tian-Todorov (BTT) theorem states that deformations of cohomologically Kähler manifolds with trivial canonical bundle are unobstructed \cite{Bolg}, \cite{Tian}, \cite{Todo}. 
More abstractly, the BTT theorem may be stated as the underlying Lie algebra of a $BV$-algebra satisfying
the so-called degeneration condition to be quasi-abelian (see \cite{Terilla}, \cite{KaKoPa}, see also \cite{Iacono} and \cite{BrLa} for generalizations to 
other algebraic structures and to $BV_\infty$-algebras respectively).
The proofs therein follow after recursively solving the Maurer-Cartan or master equation.

In this note, we explain an operadic proof of a slight generalization  of the abstract BTT Theorem, based on the formality of the operad of compactified moduli spaces of genus zero curves. Our approach allows for a partial extension to fields of positive characteristic (we refer the reader to \cite{brantnerdeformations} for a version of the BTT theorem in positive characteristics).

\begin{theo}\label{maintheo}
Let $K$ be a field. Let $A$ be a dg algebra over the framed little disks operad.
Assume that $A$ satisfies the degeneration condition. Then:
\begin{enumerate}
\item if the field $K$ is of characteristic zero, the underlying Lie algebra of $A$ is quasi-abelian.
\item if the field $K$ is of characteristic $p$, the underlying Lie algebra of $A$ is quasi-abelian up to arity $p-1$.
\end{enumerate}
\end{theo}

The degeneration condition is explained in Section \ref{sec : degeneration}. It amounts to the homotopical triviality of the action of the arity $1$ operations of the operad on the algebra. This theorem is homotopy invariant, and so it applies to an algebra over any operad that has the homotopy type of the framed little disks operad. In particular, by formality of that operad in characteristic zero, it applies to the $BV$ and $BV_\infty$ operads.

We shall also prove a ``quantum'' BTT theorem generalizing work of \cite{Terilla}, \cite{BrLa} and \cite{Tu}. In order to make sense of this theorem, we need to introduce some terminology. We shall assume that we are working over a field of characteristic zero. Let $A$ be a dg algebra over the framed little disks operad. We can restrict to arity $1$ operations and this induces an $SO(2)$-action on $A$. We denote by $A^{hSO(2)}$ the homotopy fixed points of this action. In general, it is best to think of this object $\infty$-categorically but in many cases, this can be constructed very explicitly using a negative cyclic complex (see Remark \ref{rema: negative cyclic}). It can be shown that $A^{hSO(2)}$ inherits a natural action of the operad $Grav$, a spectral version of the gravity operad (see \cite[Corollary 2.4]{westerlandequivariant}, see also \cite[Section 4]{wardmaurer}). This operad receives a map from the shifted Lie operad $\Sigma^2 Lie$.

\begin{theo}\label{maintheo2}
Let $K$ be a field of characteristic zero. Let $A$ be a dg algebra over the framed little disks operad.
Assume that $A$ satisfies the degeneration condition. 
Then the $Grav$-algebra $A^{hSO(2)}$ is quasi-abelian. In particular it is also quasi-abelian as a $\Sigma^2Lie$-algebra.
\end{theo}

\subsection*{Notations and conventions}

In this paper, we use the language of $\infty$-categories. For us an operad is an enriched $\infty$-operad in the sense of Haugseng \cite{haugsengoperads}. The enrichements that we shall consider will be either spaces, spectra or $K$ modules for $K$ a field. We denote the latter $\infty$-category by $\mathrm{Mod}_K$. It is classical that this $\infty$-category can be obtained by localizing the $1$-category of chain complexes over $K$ at the quasi-isomorphisms. 

By convention our operads always start in arity $1$.

Given a spectrum $X$ and a commutative ring $K$, we write $X\otimes K$ for the tensor $X\otimes HK$ with $HK$ the Eilenberg-MacLane spectrum of $K$. We use the same notation to denote a chain complex over $K$ that represents this homotopy type. We extend this notation to $X$ a space so that $X\otimes K:=\Sigma^{\infty}_+X\otimes K$.

Upper indices follow the cohomological grading convention while lower indices follow the homological grading convention. For us, a chain complex has homological grading and a differential of degree $-1$, a cochain complex has cohomological grading and a differential of degree $1$.

\subsection*{Acknowledgments}
The authors would like to sincerely thank the referees for their thorough reading and insightful comments, which have significantly improved the quality and clarity of this manuscript.

\section{A bestiary of operads}

Throughout, we denote by $\Lie$ the spectral Lie operad. This is defined as the Koszul dual operad to the spectral commutative operad and also as the operad of Goodwillie derivatives of the identity functor on pointed spaces (see \cite{chingbar}). We can tensor this operad with any commutative ring $R$ and obtain an operad in the symmetric monoidal $\infty$-category of modules over that ring which can be thought of as a differential graded operad over that ring. The homotopy/homology groups of $\Lie\otimes R$ are simply given by the (desuspended) algebraic Lie operad over $R$, i.e. the operad encoding the structure of a dg Lie algebra over $R$ with Lie bracket of degree $-1$ (see \cite[Example 9.50]{chingbar}). We shall denote this algebraic version by $Lie$.

There are maps
\[\Sigma^2\Lie\to \Sigma^{\infty}_+D_2\to \Sigma^{\infty}_+FD_2\]
where $D_2$ and $FD_2$ denote the (non-unital) topological little disks operad and framed little disks operad respectively. The first map is Koszul dual to the canonical map from $D_2$ to the commutative operad and uses the Koszul self-duality of $D_2$ (see \cite{fressekoszul} for a version of this theorem over a field and \cite{chingkoszul} over the sphere spectrum).

If we tensor the above maps with a commutative ring $R$ and take homotopy groups, these produce the maps of algebraic operads
\[\Sigma^2 Lie\to Ger\to BV\]
where $Ger=H_*(D_2,R)$ is the non-unital Gerstenhaber operad (i.e. the operad controlling the structure of a Gerstenhaber algebra without unit) and $BV=H_*(FD_2,R)$ is the non-unital Batalin-Vilkovisky operad. The composite of these two maps captures the fact that a $BV$-algebra has an underlying Lie algebra with a bracket of degree $1$.

We can extend this diagram of operads to the left and to the right. For this recall that the little $2$-disks operad $D_2$ has an action of the circle group $SO(2)$. This action is what allows to produce the framed little disks operad $FD_2$ as a semi-direct construction introduced in \cite{salvatorewahl}:
\[FD_2:=D_2\rtimes SO(2).\]
We can then quotient the framed little disks operad by forcing the arity $1$ space to be contractible. By the main theorem of \cite{drummondhomotopically}, we obtain the operad in topological spaces $\overline{\Mm}$ of compactified moduli spaces of curves:
\[\overline{\Mm}(n):=\overline{\Mm}_{0,n+1}\]
where we are using the convention that $\overline{\Mm}(1)=*$.

We can also take $SO(2)$ fixed points of the spectral operad $\Sigma^{\infty}_+D_2$ and obtain a spectral version of the gravity operad of Getzler (see \cite{westerlandequivariant}). We shall denote this operad by $\Grav$:
\[\Grav:=(\Sigma^{\infty}_+D_2)^{hSO(2)}.\]
We thus have the following diagram of spectral operads
\[\Sigma^2\Lie\to \Grav\to \Sigma^{\infty}_+D_2\to \Sigma^{\infty}_+FD_2\to \Sigma^{\infty}_+\overline{\Mm}.\]

After tensoring with a commutative ring $R$ and taking homotopy groups, we obtain the following diagram of algebraic operads over $R$:
\[\Sigma^2Lie\to Grav\to Ger\to BV\to Hycom\]
where $Hycom:=H_*(\overline{\Mm},R)$ is the hypercommutative operads.

\section{Quasi-abelianity}

Let $K$ be a field and $L$ be an operad in $\mathrm{Mod}_K$ that is quasi-isomorphic to $Lie\otimes K$. In particular, since $L(1)\simeq K$, there is a canonical map 
\[L\to I\]
in the homotopy category of operads in $\mathrm{Mod}_K$, where $I$ is the initial one-color operad. It follows that any $K$-module $V$ can be endowed with a trivial $L$-algebra structure that we denote as $V^{\textrm{triv}}$. This is done explicitly as follows. According to \cite[Theorem 5.12]{haugsengalgebras}, an $L$-algebra in $\mathrm{Mod}_K$ corresponds to map of $\mathrm{Mod}_K$-enriched operads $L\to \overline{\mathrm{Mod}}_K$ where $\overline{\mathrm{Mod}}_K$ denotes the $\mathrm{Mod}_K$-enriched operad underlying the closed symmetric monoidal $\infty$-category $\mathrm{Mod}_K$.   Then the $L$-algebra $V^{\textrm{triv}}$ is given by the map
\[L\to I\xrightarrow{V} \overline{\mathrm{Mod}}_K\]
where the second map classifies the unique $I$-algebra structure on $V$.

\begin{defi}\label{defi quasiabelian}
We say that an algebra $A$ over $L$ in $\mathrm{Mod}_K$ is \textit{quasi-abelian} if it satisfies one of the following equivalent conditions:
\begin{enumerate}
\item It is weakly equivalent as an $L$-algebra to one of the form $V^{\textrm{triv}}$.
\item The $\mathrm{Mod}_K$-operad map $L\to \overline{\mathrm{Mod}}_K$ classifying the algebra structure factors through $I$.
\end{enumerate}
\end{defi}

Our main theorem also involves the concept of being quasi-abelian up to a certain arity. We now explain the meaning of this phrase. An $n$-truncated operad in $\mathrm{Mod}_K$ is, by definition, a collection of $K$-modules $\{P(1),\ldots, P(n)\}$ equipped with all the operations that exist on the first $n$ terms of an operad. Equivalently, this can be viewed as the full subcategory of the category of operads in $\mathrm{Mod}_K$ spanned by objects that are zero in arity greater than $n$. Any operad $P$ in $\mathrm{Mod}_K$ admits a canonical $n$-truncation that we denote by $P_{\leq n}$. This is simply the right adjoint of the inclusion of $n$-truncated operads in operads.

\begin{defi}\label{defi quasi abelian up to arity n}
We say that an algebra $A$ over $L$ in $\mathrm{Mod}_K$ is \textit{quasi-abelian up to arity $n$} if the $n$-truncation of the  $\mathrm{Mod}_K$-operad map $L\to (\overline{\mathrm{Mod}}_K)$ classifying the $L$-algebra structure on $A$ factors through $(I)_{\leq n}$.
\end{defi}

\section{The degeneration condition}\label{sec : degeneration}

The degeneration condition is introduced and studied in 
\cite{Terilla} and \cite{KaKoPa} for $BV$-algebras (see also  \cite{BrLa} and 
\cite{DSV} for a generalization to $BV_\infty$-algebras).
 We rewrite this condition in the following general form.


Let $P$ be an operad in $\mathrm{Mod}_K$ that is weakly equivalent to $FD_2\otimes K$ (in particular, by the main result of \cite{giansiracusacyclic}, $P$ could be $BV$ or $BV_\infty$). The $K$-algebra $P(1)$ is weakly equivalent to $S^1\otimes K$ and there is an augmentation map
\[P(1)\simeq S^1\otimes K\to K\simeq *\otimes K\]
which yields, by restriction of scalars, a functor
\[\mathrm{Mod}_K\to\mathrm{Mod}_{P(1)}\]

\begin{defi}
We say that a $P$-algebra $A$ satisfies the \textit{degeneration condition} if the $P(1)$-module structure on $A$ is homotopically trivial, that is if $A$ is in the essential image of the above functor.
\end{defi}

\begin{rema}\label{rema: negative cyclic}
Observe that the degeneration condition only depends on the underlying $P(1)$-module structure and not on the full $P$-algebra structure. If $K$ is of characteristic zero and $P(1)=K[\Delta]/\Delta^2$ with $|\Delta|=1$ (for example if $P=BV$), then we may form the negative cyclic complex
\[(A[[u]],d+u\Delta)\]
where $u$ is a formal variable of degree $-2$. We can view this object as the total complex of a bicomplex with differentials $d$ and $u\Delta$. Then it can be shown that $A$ satisfies the degeneration condition if and only if the associated spectral sequence degenerates at the first page. This is also equivalent to saying that the cohomology of the negative cyclic complex with respect to the total differential $d+u\Delta$ is free as a $K[[u]]$-module (this is \cite[Definition 4.13]{KaKoPa}). This fact can be extended to $BV_\infty$-algebras by introducing a suitable version of the negative cyclic complex. We refer the reader to \cite[Remark 3.14]{BrLa} and \cite[Proposition 1.6]{DSV}.
\end{rema}

\begin{theo}\label{theo : DC}
Let $A$ be a $FD_2\otimes K$-algebra in $\mathrm{Mod}_K$. Then $A$ satisfies the degeneration condition if and only if the map
\[FD_2\otimes K\to\overline{\mathrm{Mod}}_K\]
classifying the algebra structure on $A$ can be factored through
\[FD_2\otimes K\to\overline{\Mm}\otimes K.\]
\end{theo}

\begin{proof}
By the main theorem of \cite{drummondhomotopically}, there is a pushout square
\[
\xymatrix{
FD_2(1)\otimes K\ar[r]\ar[d]& FD_2\otimes K\ar[d]\\
\ast\otimes K\ar[r]&\overline{\Mm}\otimes K
}
\]
in the $\infty$-category of operads in $\mathrm{Mod}_K$. The result follows by the universal property of pushouts.
\end{proof}

%

\section{Proof of the main results}\label{secproof}

The proof of our main theorem relies on the fact that the homotopy theory of operads in $\mathrm{Mod}_K$ can be presented by a model category (see \cite{hinichhomological} for a construction of this model structure) and that the operad $Lie\otimes K$ has an explicit cofibrant replacement. From now on, we are only considering one-color operads. By \cite[Theorem 7.3.1]{whitesmith} the model structure of operads in chain complexes over $K$ does model the $\infty$-category of operads in $\mathrm{Mod}_K$. Likewise, there is a model structure on $n$-truncated operads in chain complexes whose fibrations and weak equivalences are aritywise and which models $n$-truncated operads in $\mathrm{Mod}_K$. This follows form the fact that the colored operad whose algebras are one-color operads (or one-color $n$-truncated operads) is $\Sigma$-cofibrant.

We first recall some facts about the (twice suspended) $L_{\infty}$-operad. This is, by definition, the cobar construction of the cooperad $(\Sigma^{-2}Com)^{\vee}$. Explicitly, it is generated by a skew symmetric operation $l_n$ in arity $n$ and degree $2n-3$ for any $n\geq 2$. These operations are subject to the relations
\begin{equation}\label{relation l-infinity}
d(l_n)=\sum_{p+q=n+1}\sum_{\sigma\in Sh_{p-1,q}^{-1}} \pm\sigma.(l_p\circ_1l_q)
\end{equation}
where $Sh_{p-1,q}^{-1}\subset S_n$ denotes the set of $(p-1,q)$-unshuffles and $\sigma.x$ denotes the action of an element $\sigma$ of $S_n$ on an element $x$ of arity $n$ in an operad.

\begin{prop}\label{prop cofibrant}
If $K$ is of positive characteristic $p$,  then $(\Sigma^2L_\infty)_{\leq p-1}$ is cofibrant. The same is true without truncation if $K$ is of characteristic zero.
\end{prop}

\begin{proof}
In order to simplify notations, we write $P= \Sigma^2L_\infty$. Let us denote by $P_n$ the differential graded operad with generators $l_2,\ldots,l_n$ and all the relations of the form \eqref{relation l-infinity} that only involve those operations. Then, there are canonical map $P_{n-1}\to P_{n}$ and  $P=\mathrm{colim}_n P_n$. Moreover, the map $P_{n-1}\to P_{n}$ is obtained by attaching an operadic cell of arity $n$. This cell attachment is a cofibration as long as $K[\Sigma_n]$ is semi-simple (since in that case, any representation of $\Sigma_n$ is projective). By Maschke's theorem, this is always the case in characteristic zero and this is the case if $n<p$ in characteristic $p$. Finally, we observe that $(P_n)_{\leq n}\cong P_{\leq n}$.
\end{proof}

\begin{prop}\label{prop intrinsic formality of Lie}
Let $n$ be an integer. Let $L$ be a differential graded operad over $K$ whith an isomorphism $\alpha:\Sigma^n Lie\cong H_*(L)$. Then 
\begin{enumerate}
\item If $K$ is of characteristic zero, there is a quasi-isomorphism $\Sigma^nL_\infty\to L$ inducing $\alpha$ in homology.
\item If $K$ is of characteristic $p$, there is a quasi-isomorphism $(\Sigma^nL_\infty)_{\leq p-1}\to L_{\leq p-1}$ inducing $\alpha$ in homology.
\end{enumerate}
\end{prop}

\begin{proof}
The truth of this statement does not depend on $n$ since we can move from one case to another by applying an appropriate operadic shift. So let us do the proof in the case $n=1$. The operad $\Sigma Lie$ is concentrated in degree zero. Thus we have a zig-zag of quasi-isomorphisms of operads
\[L\xleftarrow{\sim} t_{\geq 0} L\xrightarrow{\sim} t_{\leq 0}t_{\geq 0} L = \Sigma Lie\]
where $t_{\geq 0}$ and $t_{\leq 0}$ are the standard truncation functors in chain complexes. Moreover, there is a canonical quasi-isomorphism of operads $\Sigma L_\infty\to \Sigma Lie$ sending $l_2$ to $[-,-]$ and all other $l_n$ to zero. This produces a map $\Sigma L_\infty\to L$ in the homotopy category of operads. By the previous proposition, the source is cofibrant (up to truncating in the positive characteristic case), therefore, by a standard model categorical argument, there is a single map of operads that represents the same map as this zig-zag in the homotopy category.
\end{proof}

The crux of our argument is the following operadic proposition. In order to clarify its statement, let us recall that $\Lie$ comes with an augmentation $\Lie\to I$. Given another operad $P$ over $K$, the zero map $\Lie\otimes K\to P$ is by definition the composite
\[\Lie\otimes K\to I\to P.\]

\begin{prop}\label{prop : operadic}
Let $K$ be a field. The set of homotopy classes of maps of operads 
\[(\Sigma^2 \Lie\otimes K)_{\leq n}\longrightarrow (\overline{\mathcal{M}}\otimes K)_{\leq n}\]
is a singleton (given by the zero map) in the following cases:
\begin{enumerate}
\item $K$ is of characteristic zero and $n$ is anything (including $+\infty$),
\item $K$ has positive characteristic and $n$ is smaller than the characteristic of $K$.
\end{enumerate}
\end{prop}

\begin{proof}
Thanks to the discussion at the beginning of this section, we may work in the model category of operads (or truncated operads) in chain complexes over $K$. In both cases, the target operad is formal (in the characteristic zero case, this follows from \cite{GNPR} and in the positive characteristic case from \cite[Theorem 5.3]{ciricietale}), so we can replace it by its homology which is just the hypercommutative operad $Hycom$. Moreover, by the previous proposition, in both cases, the source operad is quasi-isomorphic to the operad $\Sigma^2(L_{\infty})_{\leq n}$. 

Finally, by Proposition \ref{prop cofibrant}, the operad $\Sigma^2(L_{\infty})_{\leq n}$ is cofibrant. So we are simply calculating the set of maps $(\Sigma^2L_\infty)_{\leq n}\to Hycom_{\leq n}$ modulo homotopy. But the generators of the operad $\Sigma^2 L_\infty$ are all in odd degrees while the operad $Hycom$ is concentrated in even degrees so this set of maps is a singleton even before taking homotopy classes.
\end{proof}

\begin{prop}\label{prop : operadic2}
Let $K$ be a field of characteristic zero. Give $\overline{\Mm}$ the trivial $SO(2)$-action. The set of homotopy classes of maps of operads 
\[\Grav\otimes K\longrightarrow (\overline{\Mm}\otimes K)^{hSO(2)}\]
is a singleton (given by the zero map).
\end{prop}

\begin{proof}
The proof is analogous to the previous one. First of all, we observe that the formality of $\overline\Mm$ induces a quasi-isomorphism
\[(\overline{\Mm}\otimes K)^{hSO(2)}\simeq Hycom^{hSO(2)}\]
where both operads are given the trivial $SO(2)$-action. Moreover, we have  
\[Hycom^{hSO(2)}\simeq Hycom\otimes_K K[[u]]\] 
where $u$ is a formal variable of degree $-2$. In other words $(\overline{\Mm}\otimes K)^{hSO(2)}$ is formal and its homology is concentrated in even degrees. The operad $\Grav\otimes K$ is also formal by \cite{duponthorel}. In order to apply the previous strategy it suffices to observe that a minimal model of the algebraic operad $Grav$ is given by $\Omega(\Sigma^{-2}Hycom^{\vee})$, where $\Omega$ denotes the operadic cobar construction (see \cite[Theorem 4.6]{getzleroperads} but beware that there is a shift difference between Getzler's definition of gravity and the one that we are using here). In particular it is free on generators of odd degrees so the previous proof applies.
\end{proof}

\begin{proof}[Proof of Theorem \ref{maintheo}]
We treat the positive characteristic case. The other case is similar and easier. Let us simply write $L=\Sigma^2\Lie\otimes K$. By definition \ref{defi quasi abelian up to arity n}, an $L$-algebra $A$ is quasi-abelian up to arity $p-1$ if the classifying map 
\[L_{\leq p-1}\to (\overline{\mathrm{Mod}}_K)_{\leq p-1}\]
factors through $I_{\leq p-1}$. Using Theorem \ref{theo : DC}, we know that the map $FD_2\otimes K\to \overline{\mathrm{Mod}}_K$ classifying the $FD_2$-algebra structure on $A$ factors as
\[FD_2\otimes K\to \overline{\mathcal M}\otimes K\to \overline{\mathrm{Mod}}_K\]
But thanks to Proposition \ref{prop : operadic}, the composed map $L_{\leq p-1}\to (FD_2\otimes K)_{\leq p-1}\to (\overline{\mathcal M}\otimes K)_{\leq p-1}$ factors through $I_{\leq p-1}$ as desired.
\end{proof}

\begin{rema}\label{ddlemma}
If $(A,\wedge, d,\Delta)$ is a $BV$-algebra and the $d\Delta$-condition is satisfied, there is a much simpler proof for Theorem \ref{maintheo}. Indeed,
both the inclusion $\mathrm{Ker}(\Delta)\hookrightarrow A$ and the projection map $\mathrm{Ker}(\Delta)\to H_\Delta(A)$
are morphisms of Lie algebras, where $H_\Delta(A)$ has trivial differential and trivial Lie bracket.  Now, the 
$d\Delta$-condition 
\[\mathrm{Ker}(d)\cap \mathrm{Im}(\Delta)=\mathrm{Ker}(\Delta)\cap \mathrm{Im}(d)=\mathrm{Im}(d\Delta)\]
 ensures that these maps are quasi-isomorphisms.
\end{rema}

\begin{proof}[Proof of Theorem \ref{maintheo2}]
Using Theorem \ref{theo : DC}, the map $FD_2\otimes K\to \overline{\mathrm{Mod}}_K$ factors through $\overline{\Mm}\otimes K$. After applying $SO(2)$-fixed points, we thus get a map 
\[\Grav\otimes K=(\Sigma^{\infty}_+D_2)^{hSO(2)}\to (\Sigma^{\infty}_+\Mm)^{hSO(2)}\to \overline{\mathrm{Mod}}_K\]
classifying the $Grav\otimes K$-structure on $A^{hSO(2)}$. Thanks to Proposition \ref{prop : operadic2}, the first map factors through $I$ after tensoring with $K$. It follows that the composed map factors through $I$ as well. The last part of the statement of Theorem \ref{maintheo2} simply follows from the fact that the map 
\[\Sigma^2\Lie\to \overline{\mathrm{Mod}}_K\]
classifying the $\Sigma^2\Lie$-structure on $A^{hSO(2)}$ factors through $\Grav\otimes K$.
\end{proof}

\section{The classical BTT Theorem}
The infinitesimal deformations of a complex manifold $M$ are governed by the Kodaira--Spencer Lie algebra
\[\Ll^{*}:=\mathcal{A}^{0,*}(\Theta)=\Gamma(M,\Lambda^*\overline{\Theta}^\vee\otimes \Theta),\]
defined by taking global sections of the sheaf of complex differential forms of type $(0,*)$ with coefficients in the holomorphic tangent bundle $\Theta$ of $M$. The differential on $\Ll^*$ is induced by the Dolbeault operator 
$\overline{\partial}$ and the Lie bracket is induced by the Lie bracket of vector fields.

A main objective in deformation theory is to integrate first-order
deformations: given a cohomology class $\alpha\in H^{1}(M,\Theta)$, the problem is to find a formal power series 
\[\xi(t)=\sum_{i\geq 1} \xi_it^i\text{ with }\xi_i\in \Ll^{1}\]
satisfying the Maurer-Cartan equation 
\[\overline{\partial}\xi(t)+{1\over 2}[\xi(t),\xi(t)]=0\]
and such that $[\xi_1]=\alpha$. This gives a recursive system of equations 

\begin{equation}\label{system}\tag{$\star$}
\left\{
\arraycolsep=4pt\def\arraystretch{1.4}
\begin{array}{l}
\overline{\partial} \xi_1=0\\
\overline{\partial} \xi_2+{1\over 2}[\xi_1,\xi_1]=0\\
\overline{\partial} \xi_3+[\xi_1,\xi_2]=0\\
\cdots
\end{array} \right.
\end{equation}

 If this system has a solution for any choice of $\alpha$ then it is said that the 
deformation theory is \textit{unobstructed}. This is always the case if $H^2(M,\Theta)=0$, but not in general.
One sufficient condition for a deformation theory to be unobstructed is when it is governed by a quasi-abelian dg-Lie algebra. This follows from the well-known fact that the deformation functors associated to quasi-isomorphic Lie algebras (over a field of characteristic zero) are isomorphic and explains the connection between the original BTT Theorem and Theorem \ref{maintheo}.

\begin{theo}[Classical BTT]\label{theo : classical btt}
 The deformation theory of any compact Calabi--Yau manifold is unobstructed.
\end{theo}
\begin{proof}The Calabi--Yau condition ensures the existence of a nowhere vanishing holomorphic $m$-form $\Omega$, where $m$ is the complex dimension of the manifold.
 Contraction with the form $\Omega$ gives a natural isomorphism 
$\eta:\Ll^{p,q}\to \Aa^{m-p,q}$, where 
\[\Ll^{p,q}:=\Aa^{0,q}(\Lambda^p\Theta)=\Gamma(M,\Lambda^q \overline{\Theta}^\vee\otimes \Lambda^p\Theta).\]
The exterior product makes the bigraded vector space $\Ll^{*,*}$ into a commutative algebra and 
the Dolbeault operator induces a differential $\overline\partial$ of bidegree $(0,1)$.
Define an operator $\Delta$ on $\Ll^{*,*}$ of bidegree $(-1,0)$ by letting 
\[\Delta:=\eta^{-1} \circ \partial\circ \eta.\]
The classical Tian-Todorov Lemma states that the relation 
\[(-1)^{p}[\alpha,\beta]=\Delta(\alpha\wedge\beta)-\Delta(\alpha)\wedge\beta-(-1)^{p+1}\alpha\wedge\Delta(\beta)\]
holds for all $\alpha\in \Ll^p$ and $\beta\in \Ll^q$ (see for instance \cite{Huy}, Lemma 6.1.9). An extension of this equation to elements in $\Ll^{p,q}$, proven by Barannikov and Kontsevich in \cite{BaKo}, implies that the tuple
\[(\Ll^{*,*},\overline{\partial},\wedge,\Delta)\]
is a BV-algebra. In this case, the Lie bracket of $\Ll^{*,*}$ has bidegree $(-1,0)$ and is defined by the Schouten-Nijenhuis bracket for multi-vector fields 
and exterior products of forms. In particular, the Kodaira--Spencer Lie algebra $\Ll^*=\Ll^{1,*}$ is a direct summand (as a dg Lie algebra).

Lastly, note that the $\partial\overline{\partial}$-condition for compact Kähler manifolds
implies  that the bicomplex $(\Ll^{*,*},\overline{\partial},\Delta)$ satisfies the
$d\Delta$-condition. Therefore by Remark \ref{ddlemma}, $\Ll^{*,*}$ is a quasi-abelian Lie algebra.
Since  $\Ll^{*}=\Ll^{1,*}$ is a direct summand of $\Ll^{*,*}$, 
it follows that $\Ll^*$ is also quasi-abelian.
\end{proof}

\begin{rema}
 The original approach of \cite{Tian}, \cite{Todo} to prove the BTT doesn't use the extended BV-algebra $\Ll^{*,*}$ introduced in this proof. Instead,
 the system of equations (\ref{system}) is recursively solved using both the Tian-Todorov Lemma and the $\partial\overline\partial$-condition (see for instance \cite{Huy}).
\end{rema}

\begin{rema}[Extended deformation theory]
While the Kodaira--Spencer Lie algebra $\Ll^*$ governs the deformation theory of the underlying complex structure $J:T_M\to T_M$,
as shown by Gualtieri \cite{Gualtieri}, the extended Lie algebra $\Ll^{*,*}$ used in the above proof governs deformations of $M$ as families of generalized complex structures
$J:T_M\oplus T_M^\vee\to T_M\oplus T_M^\vee$, in the sense of Hitchin. In particular, for a Calabi--Yau manifold, the above proof also shows that the deformation theory of a Calabi--Yau manifold in the generalized sense is also unobstructed.
\end{rema}

\section{Some more BV-algebras of geometric origin}In this section we collect some  commutative $BV_\infty$-algebras that arise from geometric structures and satisfy the degeneration property. In particular, Theorem \ref{maintheo} applies.
The operad governing commutative $BV_\infty$-algebras is a partial cofibrant resolution of $BV$, where only the operation $\Delta$ is resolved (but
not the commutative structure). Still, it receiveds a map from $L_\infty$ and so every commutative $BV_\infty$-algebra has an associated $L_\infty$-algebra structure (see Definition 3.17 of \cite{BrLa} for an explicit formula).

Let $(A,\wedge,d)$ be a commutative dg algebra and assume $\Lambda\in \mathrm{End}(A)$ is an operator of degree $-2$ and order $\leq 2$. For all $k\geq 1$, let 
\[\Delta_k={1\over k!}\underbrace{[\Lambda,[\Lambda,\cdots[\Lambda}_{k},d]]].\]
Then the tuple $(A,\wedge,d,\Delta_k)$ is a commutative $BV_\infty$-algebra satisfying the degeneration condition (see \cite{CiWi}) and so by Theorem \ref{maintheo} the associated $L_\infty$-algebra is quasi-abelian.
Let us list a few geometric examples where this situation applies:

\begin{itemize}[leftmargin=*]
 \setlength\itemsep{.5em}
 \item (Poisson). On the de Rham algebra $\Aa_{\dR}(M)$ of a Poisson manifold $M$, consider 
 the operator $\Lambda:=i_\pi$ defined by contracting forms 
with the Poisson bi-vector field $\pi$. We obtain the BV-algebra of Koszul \cite{Koszul}, with $\Delta=\Delta_1=[i_\pi,d]$ and $\Delta_k=0$ for $k\geq 2$. By Theorem \ref{maintheo}, the underlying Lie algebra is quasi-abelian. In particular, it is formal and so we recover the main result of \cite{ShTa} (see also \cite{FiMa}).
Note that the de Rham algebra of $M$ is not formal as a commutative algebra in general.

\item (Hermitian). Let $M$ be a Hermitian manifold with almost complex structure $J$ and fundamental $(1,1)$-form $\omega$. Letting $\Lambda$ to be the adjoint to the Lefschetz operator we obtain a $BV_\infty$-algebra structure on $\Aa_{\dR}(M)$ with 
\[\Delta_1=-(d_c^*+T_c^*),\, \Delta_2=d\omega_c^*,\, \text{ and }\Delta_i=0\text{ for }i\geq 3.\]
Here $d_c^*$ is the formal adjoint to $d_c=J^{-1} d J$, $d\omega_c^*$ is the formal adjoint to $d\omega_c=J^{-1} d\omega J$ and $T_c^*$ is the adjoint to $T_c=J^{-1}[\Lambda, d\omega] J$. There is also a complex version for the Dolbeault algebra as well as a generalization of this construction for almost Hermitian manifolds (see \cite{CiWi} for details).

\item (Kähler). In the case of Kähler manifolds, the above two examples converge to give the same BV-algebra:
taking 
 again $\Lambda$ to be the adjoint of the Lefschetz operator we get that $\Delta_1=-d_c^*$ and $\Delta_i=0$ for all $i\geq 2$.
There is also a complex version of this BV-algebra with differential $\overline\partial$ and $\Delta_1=-i\partial^*$. In both cases, the $d\Delta$-condition is satisfied and one gets formality of the de Rham algebra as a commutative, as a Lie and as a BV-algebra. Note that, in the Kähler case, the hypercommutative algebra arising from a canonical trivialization determined by the Kähler metric is also formal, as we show in \cite{CiHoBV}.
\end{itemize}

The next two examples do not exactly fit into the same framework as the above, but the construction is of similar nature:
\begin{itemize}[leftmargin=*]
 \setlength\itemsep{.5em}
 \item (Generalized Poisson) Consider a generalized Poisson structure on a manifold $M$, of the form $\pi=\pi_1+\pi_2+\cdots +\pi_d$, where $\pi_i$ is a $(i+1)$-multi-vector field. The operators $\Delta_k=[i_{\pi_k},d]$ define a $BV_\infty$-algebra structure on the de Rham algebra $\Aa_{\dR}(M)$ of $M$, satisfying the degeneration condition (see \cite{BrLa}).
 In the case $\pi=\pi_1$, this recovers Koszul's BV-algebra.

\item (Jacobi) A Jacobi structure on a smooth manifold $M$ is given by a bi-vector field $\pi$ and a vector field $\eta$ satisfying $[\pi,\pi]=2\eta\pi$ and $[\pi,\eta]=0$. The operators $\Delta_1=[i_\pi,d]$ and $\Delta_2=-i_\eta i_\pi$ and $\Delta_i=0$ for all $i\geq 3$, define a $BV_\infty$-algebra which satisfies degeneration (see \cite{DSV}, see also \cite{GuMu} for an abstract notion of Jacobi algebra). Again, when $\eta=0$ this recovers Koszul's BV-algebra for Poisson manifolds.

\end{itemize}

\section{The non-commutative BTT Theorem}

In the non-commutative setting, a BTT theorem for proper Calabi--Yau differential graded (dg) algebras was sketched in
\cite{KaKoPa} (see also \cite{Terilla,Iwa0,Tu,iwanari, TVVB}). 
In this section, we fix a field $K$ of characteristic zero. For a dg category over $K$, we denote by $\CH^*(\Cc)$ the cohomological Hochschild complex of $\Cc$ and by $\CH_*(\Cc)$ the homological Hochschild complex of $\Cc$. We denote by $\Cyc_{-}^*(\Cc)$ the negative cyclic complex of $\Cc$. Recall (see \cite[Theorem 2.1]{hoyoishomotopy}) that
\[\Cyc_{-}^*(\Cc)\simeq \CH_*(\Cc)^{hSO(2)}\]
where the $SO(2)$-action is induced by the Connes boundary on the Hochschild complex.

A Calabi--Yau structure of dimension $n$ on a dg category $\Cc$ over $K$ is a cohomology class
\[K[n]\to \Cyc_{-}^*(\Cc)\]
such that the induced class
\[\theta:K[n]\to\Cyc_{-}^*(\Cc)\to \CH_*(\Cc)\]
induces a quasi-isomorphism
\[\CH^*(\Cc)\xrightarrow{-\cap\theta} \CH_*(\Cc)[-n]\]
When this is the case, the $D_2$ structure on $\CH^*(\Cc)$ given by the Deligne conjecture and the $SO(2)$-action on $\CH_*(\Cc)$ fit together into a $FD_2$ structure on $\CH^*(\Cc)\simeq \CH_*(\Cc)[-n]$ by the main result of \cite{costellotopological, bravcyclic}.

\begin{rema}
If $X$ is a Calabi--Yau variety, then it can be shown that the dg category $\mathrm{Perf}(X)$ of complexes of perfect $\mathcal{O}_X$-modules has a Calabi--Yau structure. Moreover, the Hochschild--Kostant--Rosenberg theorem shows that, for a smooth and proper complex variety, there is a quasi-isomorphism
\[\CH^*(\mathrm{Perf}(X))\simeq\RR\Gamma(X,\mathrm{Sym}(\Theta[1]))\]
The right hand side can be computed by the explicit complex given in the proof of Theorem \ref{theo : classical btt}. In the Calabi--Yau case, both complexes admit a $FD_2$-algebra structure and the two structures should coincide although we do not know a reference for this fact. The fact that the two Lie structures coincide is sketched in \cite[Section 3.3]{barannikovnon}. This observation shows that the commutative BTT theorem is a particular case of the non-commutative BTT theorem.
\end{rema}

\begin{theo}
Let $\mathcal{C}$ be a smooth and proper Calabi--Yau dg category over a field $K$ of characteristic zero. Then the Hochschild complex $\CH^*(\mathcal{C})$ is quasi-abelian as a $\Sigma^2\Lie\otimes K$-algebra.
\end{theo}

\begin{proof}
The homotopy fixed point spectral sequence for the Hochschild complex (which is often called the Hodge-to-de Rham spectral sequence in this context) collapses by a theorem of Kaledin (see \cite{kaledinnc} and \cite{mathewkaledin}). This implies that the degeneration propery holds for the Hochschild complex (see the beginning of Section 3 of \cite{mathewkaledin}). So this result is simply an application of Theorem \ref{maintheo}.
\end{proof}

There is also a quantum version of the non-commutative BTT theorem given by the following.

\begin{theo}
Let $\mathcal{C}$ be a smooth and proper Calabi--Yau dg category over a field $K$ of characteristic zero. Then the negative cyclic complex $\Cyc_{-}^*(\mathcal{C})[-n]$ is quasi-abelian as a $\Grav\otimes K$-algebra and in particular as a $\Sigma^2 \Lie\otimes K$-algebra.
\end{theo}

\begin{proof}
This is a direct application of Theorem \ref{maintheo2}.
\end{proof}

\begin{rema}
In \cite{TVVB}, the authors give a deformation theoretic interpretation of the Lie algebra $\Cyc_{-}^*(\mathcal{C})$ at least if $\mathcal{C}=A$ is a Calabi--Yau algebra. They show that it controls deformations of the pair $(A, \theta)$ where $\theta$ is the Calabi--Yau structure as defined above.

It would be interesting to give a conceptual reason for the fact that this Lie algebra is in fact a $Grav$-algebra. One may expect that it means that deformations of the pair $(A,\theta)$ can be defined over $\overline{\mathcal{M}}$-algebras.
\end{rema}

\bibliographystyle{amsalpha}
\bibliography{biblio}

\end{document}